\documentclass[12pt,a4paper]{article}
\usepackage{amsmath,amssymb,amsthm}
\usepackage{graphicx}
\usepackage{amsfonts,graphicx, amsthm,amsmath,amssymb}
\usepackage{hyperref}
\newtheorem{theorem}{Theorem}
\newtheorem{lemma}{Lemma}
\newtheorem{proposition}{Proposition}

\theoremstyle{definition}
\newtheorem*{remark}{Remark}
\newtheorem{definition}{Definition}

\newcommand{\eps}{\varepsilon}
\DeclareMathOperator*{\mes}{mes}
\DeclareMathOperator*{\osc}{osc}

\newcommand{\abs}[1]{\lvert#1\rvert}
\newcommand{\norm}[1]{\lVert#1\rVert}
\newcommand{\calI}{{\mathcal{I}}}

\begin{document}
\date{}
\title{Asymptotic properties of Arnold tongues and Josephson effect}
\author{
A.~Klimenko
\thanks{Steklov Mathematical Institute of RAS, Email: klimenko05@mail.ru} , O.~Romaskevich
\thanks{National Research University Higher School of Economics, \'{E}cole Normale Superieure de Lyon, Email: olga.romaskevich@gmail.com}
\thanks{Supported by part by RFBR grants
12-01-31241-mol-a and 12-01-33020-mol-a-ved}}
\maketitle
\begin{flushright}
\emph{To our dear teacher Yu.S. Ilyashenko on his 70-th birthday}
\end{flushright}
\begin{abstract}
A three-parametrical family of ODEs on a torus arises from a model of Josephson effect in a resistive case when a Josephson junction is biased by a sinusoidal microwave current. We study asymptotics of Arnold tongues of this family on the parametric plane (the third parameter is fixed) and prove
that the boundaries of the tongues are asymptotically close to Bessel functions.
\end{abstract}

\section{Introduction}
\label{sec:intro}
We will deal with a family of differential equations on a circle $\mathbb R/2\pi\mathbb Z$
\begin{equation}\label{eq:Joseph}
\frac{dx}{dt}=\frac{\cos x +a + b \cos t}{\mu},
\end{equation}
which arises in physics of the Josephson effect\footnote{In physical literature this equation is often written with sines instead of cosines. Substitutions $x\to x\pm\pi/2$, $t\to t\pm\pi/2$ transform one variant to another; one can see that these substitutions do not affect all results of this paper.}. In the paper we refer to \eqref{eq:Joseph} as to \emph{the Josephson equation}.

Here $a,b\in\mathbb R$, and $\mu>0$ are parameters. Such a family was studied in the context of Prytz planimeter \cite{Foote} as well as in the context of bicycle track trajectories \cite{Finn, LT}. For the first time the techniques of slow-fast systems (for $\mu\ll 1$) were applied to this equation by J.~Guckenheimer and Yu.~Ilyashenko in \cite{IG} but in the context of Josephson effect the family \eqref{eq:Joseph} has not been studied from a mathematical point of view till the series of works \cite{BKTcon, buch2008} by V.M. Buchstaber, O.V. Karpov and S.I. Tertychnyi. Now this subject has become quite popular, see for instance \cite{Buch, buch2012, IRF, KRS, oneline, BT}.

The family \eqref{eq:Joseph} can be generalized to the following form:
\begin{equation}\label{eq:general}
\frac{dx}{dt}=\frac{f(x)+a+bg(t)}\mu,
\end{equation}
where $f$ and $g$ are $2\pi$-periodic functions with zero averages:
\begin{equation}\label{eq:fg-averages}
\int_0^{2\pi}f(x)\,dx=0,\qquad \int_0^{2\pi}g(t)\,dt=0.
\end{equation}
Any equation of the form \eqref{eq:general} defines a vector field on a two-dimensional torus $\mathbb R^2/2\pi\mathbb Z^2$ with coordinates $x$ and $t$. Namely, introducing a new time variable $\tau$, we can express this vector field as
\begin{equation}\label{eq:Joseph_torus}
\left\{
\begin{aligned}
\dfrac{\partial x}{\partial \tau}&{}=f(x)+a+bg(t),\\
\dfrac{\partial t}{\partial \tau}&{}=\mu.
\end{aligned}
\right.
\end{equation}
The same vector field can also be considered as a vector field on a cylinder $\mathbb R^2/((x,t)\sim (x,t+2\pi))$. In both cases the  Poincar\'e map from the transversal line $\{t=0 \bmod 2\pi\}$ to itself can be defined, we denote it as $P_{a, b, \mu}$ for the torus, and $\widetilde P_{a, b, \mu}$ for the cylinder.
Clearly, $\widetilde P_{a, b, \mu}$ is a lift of $P_{a, b, \mu}$.

Consider the \emph{rotation number} $\rho_{a,b,\mu}$ of the map $P_{a,b, \mu}$ which is, by definition, a limit
\begin{equation*}
\rho_{a,b,\mu}:=\lim_{n \rightarrow \infty} \frac{\widetilde{P}_{a,b,\mu}^{\circ n}(x)-x}{2 \pi n}.
\end{equation*}
It is well known that this limit exists and does not depend on the point $x \in \mathbb{R}$ (see, for example, \cite{KH}). The value of the rotation number is an important characteristic of a map $P_{a,b, \mu}$: for instance, it is invariant under conjugation by homeomorphisms.

\begin{definition}
We say that the \emph{phase lock} occurs for the value $k \in \mathbb{R}$ of rotation number if the level set
$$E_k:=\left\{(a,b,\mu)| \rho_{a,b,\mu}=k\right\}$$
 in the space of parameters  $\mathbb{R}^2 \times \mathbb{R}_{+}$ has nonempty interior.
In this case the level set $E_k$ is called an \emph{Arnold tongue}.
\end{definition}

The structure of Arnold tongues for the equation \eqref{eq:Joseph} and its generalizations is of a great interest for physical applications as well as from a purely mathematical point of view. We study sections of Arnold tongues by the planes with fixed $\mu$. Nevertheless, we still take $\mu$ into account, particularly, constants in the $O(\,\cdot\,)$'s do not depend on $\mu$.

Since the right-hand side of the equation \eqref{eq:general} (and thus the map $\widetilde P_{a, b, \mu}$) grows monotonically with~$a$ there is no phase lock for $k \notin \mathbb{Q}$. This happens generically to Arnold tongues: they are absent for irrational values of rotation number. Moreover, the specificity of the equation \eqref{eq:Joseph} gives that for $k\in\mathbb Q\setminus\mathbb Z$ there is no phase lock as well.

It's easy to see that the substitution $u=\tan \frac{x}{2}$ conjugates the equation \eqref{eq:Joseph} to a Riccati equation. This fact was noticed by R.Foote in \cite{Foote} in the context of Prytz planimeter then rediscovered independently by Yu.Ilyashenko \cite{lectures, IRF} and V.Buchstaber, O.Karpov, S.Tertychnyj \cite{Buch} in the context of Josephson effect. This simple but important remark gives that the Poincar\'{e} map $P_{a,b, \mu}$ is conjugated to a M\"obius transformation. Lots of uncommon properties of Josephson equation follow from this fact, the absence of phase lock for non-integer rotation numbers as one of the examples.

Indeed, if $\rho_{a,b,\mu}=p/q$, $q>1$, then the Poincar\'e map has a periodic point of period $q$. But a M\"obius transformation with periodic non-fixed points should be periodic itself. Therefore, $(\widetilde P_{a,b,\mu})^q(x)=x+p$. Monotonicity in $a$ yields that this identity can appear for only one value of $a$ provided $b$ and $\mu$ are fixed, hence the level set has empty interior.

So for a fixed $\mu$ there is a countable number of tongues on the plane of parameters $(a,b)$, corresponding to integer rotation numbers. From now on we will consider the half-plane $b>0$; another half-plane could be studied using symmetries of the equation.

The previous argument uses only the fact that $f(x)=\cos x$, imposing no conditions on $g(t)$. But when $g$ is even (in particular, when $g(t)=\cos t$) the equation \eqref{eq:Joseph_torus} possesses an additional symmetry: the map $(x,t) \mapsto (-x,-t)$ brings phase curves to themselves with orientation reversed. This means that $-P_{a,b,\mu}(-x)=P^{-1}_{a,b,\mu}(x)$. Hence, if $x_0$ is fixed point of $P_{a,b,\mu}$, then $-x_0$ is also a fixed point. If the point $(a,b,\mu)$ lies on the boundary of an Arnold tongue then the M\"obius map $\widetilde P_{a,b,\mu}$ is either parabolic or identity. In the parabolic case its only fixed point $\hat x$ should satisfy $\hat x\equiv -\hat x\pmod{2\pi}$, hence $\hat x$ is either $0$ or $\pi$.

For any fixed $b$ and $\mu$ the set $E_k^{b,\mu}=\{a \in \mathbb{R}: (a,b,\mu) \in E_k\}$ is a closed interval $E_k^{b,\mu}=[a^-_{b,\mu},a^+_{b,\mu}]$. When $a$ varies from left end $a^-_{b,\mu}$ of the interval $E_k^{b,\mu}$ to its right end $a^+_{b,\mu}$,
the set $\{x:\widetilde P_{a,b,\mu}(x)>x+k\}$ grows monotonically since the right-hand side of the equation \eqref{eq:general} is monotonic in $a$.
Thus, if $\hat x$ is a fixed point of $P_{a,b,\mu}$ for $a=a^-_{b,\mu}$, then $\widetilde P_{a,b,\mu}(\hat x)>\hat x+k$ for all $a\in E_k^{b,\mu}$, except $a=a^-_{b,\mu}$, and $\hat x$ can not be a fixed point of $P_{a,b,\mu}$ for $a=a^+_{b,\mu}$.
Hence $P_{a,b,\mu}$ has a fixed point $0$ at one end of the segment $E_k^{b,\mu}$ and $\pi$ on its other end.

Therefore, for a fixed $\mu$, boundary of the Arnold tongue with rotation number equal to $k\in\mathbb Z$
can be presented as a union of two graphs of analytic functions denoted by $a_{0,k}(b)$ and $a_{\pi,k}(b)$, where
$0$ (respectively, $\pi$) is fixed by Poincar\'e map when $a=a_{0,k}(b)$ (respectively, $a=a_{\pi,k}(b)$).
These graphs can intersect, and the Poincar\'e map $\tilde{P}_{a,b,\mu}$ is identical at the intersection points.

\section{Main results}

We are interested in the asymptotics of the boundaries $a_{0,k}(b)$ and $a_{\pi,k}(b)$ of Arnold tongues for \eqref{eq:Joseph} as $b\to\infty$.
These estimates will be established in two steps. First, in Theorem \ref{thm:asymp-const} we show that the boundaries $a_{0,k}(b)$ and $a_{\pi,k}(b)$ are close to the line $a=k \mu$. Thereupon we show in Theorem \ref{thm:asymp-bessel} that the functions
$a_{0,k}(b)-k\mu$ and $a_{\pi,k}(b)-k\mu$ are asymptotically close to normalized integer Bessel functions. This fact was noticed for the first time in \cite{HollyJ}, right after the discovery of the Josephson effect in 1962 with the first explanation on a physical level of rigor; see also chapter~5 in \cite{LiUl}, \S11.1 in \cite{BaPa}, and \cite{Buch}. In this paper we give a complete proof of this statement, as well as the estimates on the difference.

\begin{theorem}\label{thm:asymp-const}There exist positive constants $C_1,C_2,K_1,K_2$ such that the following holds.

If the parameters $a,b,\mu$ are such that
\begin{equation}\label{eq:thm-const-cond}
|a|+1\le C_1\sqrt{b\mu},\qquad
b\ge C_2 \mu
\end{equation}
then
\begin{equation}\label{eq:asymp-const}
\left|\frac a\mu-\rho_{a,b,\mu}\right|\le \frac{K_1}{\sqrt{b\mu}}+\frac{K_2}{b\mu}\ln\biggl(\frac{b}{\mu}\biggr)\le
\frac{K_1}{\sqrt{b\mu}}+\frac{2K_2}{\sqrt{b\mu^3}}.
\end{equation}
\end{theorem}

\begin{theorem}\label{thm:asymp-bessel}There exist positive constants $C_1',C_2',K_1',K_2',K_3'$ such that the following holds.

For the parameters $b,\mu$ and a number $k \in \mathbb{Z}$ satisfying inequalities
\begin{equation}\label{eq:thm-bessel-cond}
|k\mu|+1\le C_1'\sqrt{b\mu},\qquad
b\ge C_2' \mu
\end{equation}
the following estimates hold
\begin{equation}\label{eq:asymp-bessel}
\begin{aligned}
\left|\frac {a_{0,k}(b)}{\mu}-k+\frac{1}{\mu}J_k\left(-\frac{b}{\mu}\right)\right|&{}\le \frac{1}{b}\biggl(K_1'+\frac{K_2'}{\mu^3}+K_3'\ln\biggl(\frac{b}{\mu}\biggr)\biggr),\\
\left|\frac {a_{\pi,k}(b)}{\mu}-k-\frac{1}{\mu}J_k\left(-\frac{b}{\mu}\right)\right|&{}\le \frac{1}{b}\biggl(K_1'+\frac{K_2'}{\mu^3}+K_3'\ln\biggl(\frac{b}{\mu}\biggr)\biggr).
\end{aligned}
\end{equation}
\end{theorem}

Theorem \ref{thm:asymp-bessel} is our main result -- it shows how the boundaries of Arnold tongues could be approximated by Bessel functions if $b$ is sufficiently large; this is illustrated by Figure \ref{fig:graph}.

\bigskip

Recall that the Bessel function of the first kind can be defined as
\begin{equation}\label{eq:bessel-def}
J_k(-z)=\frac{1}{2\pi} \int_0^{2\pi} \cos (kt+z \sin t) dt.
\end{equation}
It has the following asymptotics for large $z$ (see \cite{W}):
\begin{equation*}
J_k(-z)=\sqrt{\frac{2}{\pi z}}\cos\biggl(-z-\frac{k\pi}{2}+\frac{\pi}{4}\biggr)+O\biggl(\frac1{z^{3/2}}\biggr)\quad\text{as $z\to+\infty$}.
\end{equation*}
Applying this to \eqref{eq:asymp-bessel} we obtain
\begin{equation}\label{eq:asymp-bessel-expanded}
a_{\dots,k}(b)=k\pm\sqrt{\frac{2}{\pi b\mu}}\cos\biggl(\frac{b}{\mu}-\frac{k\pi}{2}+\frac{\pi}{4}\biggr)+O_\mu(b^{-1}\ln b).
\end{equation}
(Here $O_\mu(\,\cdot\,)$ is $O(\,\cdot\,)$ with the constant depending on $\mu$.) Therefore, the Bessel asymptotics is indeed the main term for $a_{\dots,k}(b)$.
In particular, \eqref{eq:asymp-bessel-expanded} means that the graphs of $a_{0,k}(b)$ and $a_{\pi,k}(b)$ do have infinitely many intersections, that is, each Arnold tongue has infinitely many horizontal sections of zero width. The points $(a,b)$ on the plane of parameters corresponding to the intersections of the boundaries of some Arnold tongue are clearly very special. Poincar\'{e} map $P_{a,b, \mu}$ corresponding to such points is an identity map.

\begin{definition}
Point $(a,b) \in \mathbb{R}^2$, $b \neq 0$ on the boundary of the Arnold tongue with $\rho_{a,b,\mu}=k \in \mathbb{Z}$ is called an \emph{adjacency point} if it lies on the intersection of the boundaries, i.e. $a=a_{0,k}(b)=a_{\pi,k}(b)$.
\end{definition}

\begin{figure}
\begin{center}
\includegraphics*[scale=1]{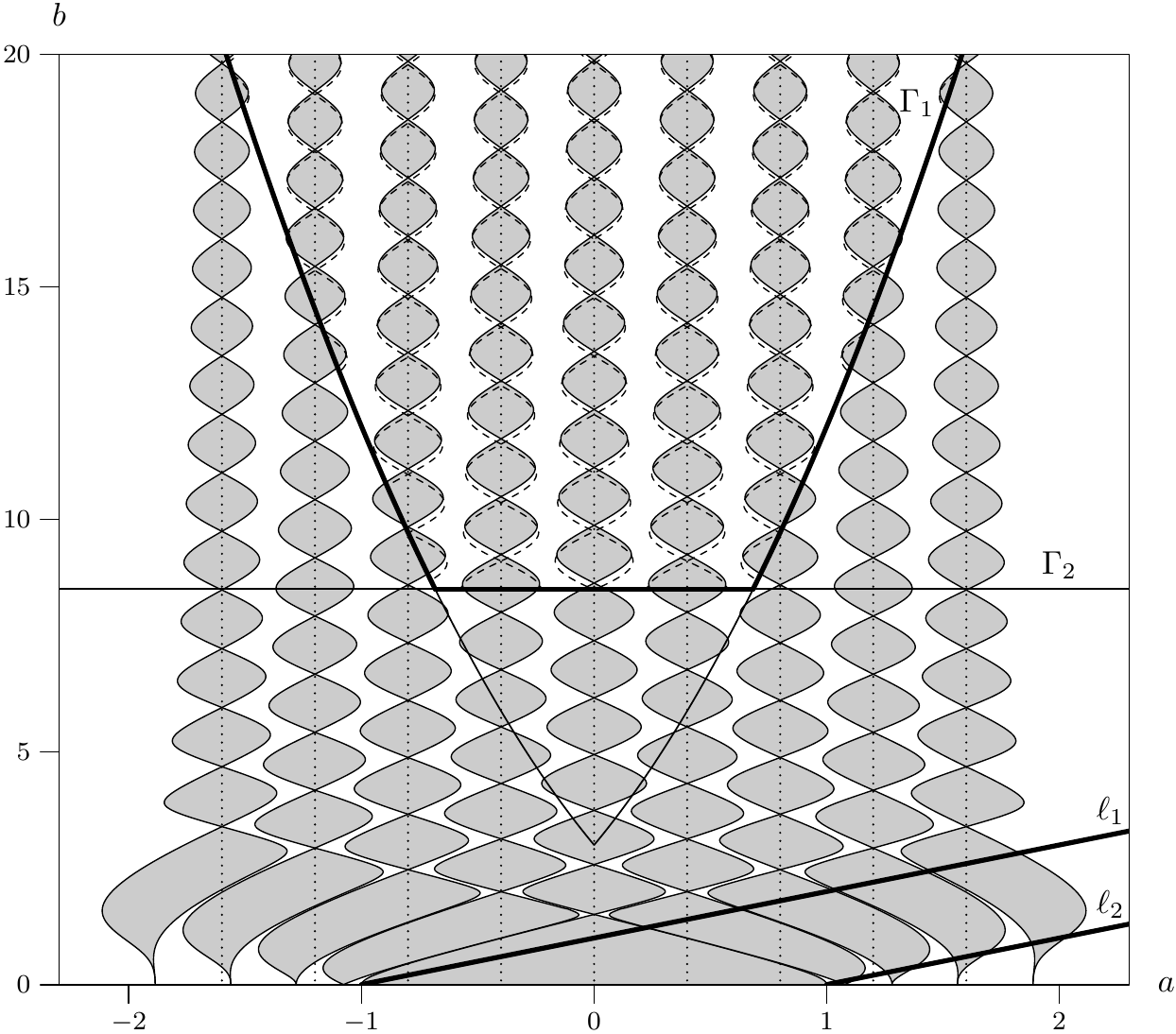}
\caption[]{Arnold tongues of Josephson equation in the domain on the plane of parameters $(a,b)$ for a fixed $\mu=0.4$.
\smallskip

Grey domains are Arnold tongues $E_k$ for $k=-4,\dots,4$, their boundaries (solid lines) are
curves $a=a_{0,k}(b)$ and $a=a_{\pi,k}(b)$.

Curves $\Gamma_1$ and $\Gamma_2$ are defined by conditions of the form \eqref{eq:thm-const-cond}.
The estimates of Theorem~\ref{thm:asymp-bessel} are applicable in the domain above both $\Gamma_1$ and $\Gamma_2$ (contoured with bold line).  Relations between conditions of the form \eqref{eq:thm-const-cond} and the conditions \eqref{eq:thm-bessel-cond} of Theorem~\ref{thm:asymp-bessel} are discussed in the first part of the proof of this theorem. Dashed lines in this domain represent Bessel approximations given by Theorem~\ref{thm:asymp-bessel}.

Dotted lines are the lines $a=k\mu$, which contain all adjacency points of Arnold tongues \cite{oneline}.

Domain between the lines $\ell_1$ and $\ell_2$ is where the slow-fast techniques work \cite{KRS}.

The computer simulations for this picture were done by I.Schurov.}\label{fig:graph}
\end{center}
\end{figure}

Recently many interesting results on the structure of Arnold tongues for the Josephson equation were discovered.
Here we present a brief summary.

First of all, the $k$-th Arnold tongue $E_k$ intersect the line $b=0$ at one point $(\operatorname{sgn} k\cdot \sqrt{k^2 \mu^2 +1}, 0)$ if $k\ne 0$, and $E_0$ intersects this line by the segment
$[-1,1]$ (for $b=0$ the equation \eqref{eq:Joseph} does not depend on time and can be easily integrated).
As we have noted above, Theorem~\ref{thm:asymp-bessel} implies that each tongue has infinitely many adjacency points.

What is more surprising, Figure \ref{fig:graph} suggests the following conjecture: the adjacency points of a $k$-th Arnold tongue lie on the same line $a(b) \equiv k \mu$ (dotted lines on Fig.~\ref{fig:graph}). This is proven in \cite{oneline} for $\mu \geq 1$ and the proof uses the classical theory of non-autonomous linear equations of complex variable.
For $\mu<1$ this fact is not yet proven and rests a reasonable conjecture. The difficulty resides in a study of adjacency points near the line $b=0$.

The result in \cite{oneline} is the only global non-trivial result on the structure of Arnold tongues of \eqref{eq:Joseph}, other results concentrate attention on the behaviour in some domains of the parameter plane.

For instance, for $\mu$ small enough the techniques of slow-fast systems can be used to show that the domain between the lines  $\ell_1=\{b=a+1\}$ and $\ell_2=\{b=a-1\}$ is filled up tightly with Arnold tongues and the distances between the tongues diminish exponentially in $\mu$. For the review of slow-fast techniques for \eqref{eq:Joseph} see \cite{KRS}.

The overall picture of the behaviour of Arnold tongues is the following: in any finite domain around the line $b=0$ the tongues fill up tightly the space \cite{KRS} and for $b$ tending to infinity (when ``$b$ is bigger than $\mu$ is smaller'') Bessel behaviour prevails over the slow-fast one. This is, however, very sketchy, and many questions still can be asked on a local behaviour of Arnold tongues.
As an example, it seems from the picture that the right boundaries of Arnold tongues $E_k$, $k>0$, have inflection points on the line $\ell_2$. We have no idea if it is true or not and how to prove it.
\medskip

Now let us sketch proofs of Theorem \ref{thm:asymp-const} and Theorem \ref{thm:asymp-bessel}. First of all, we rewrite \eqref{eq:Joseph} as an integral equation
\begin{equation}\label{eq:integral}
x(t)-x(0)=\frac{at+b\sin t+\int_0^t\cos x(\tau)\,d\tau}{\mu}
\end{equation}
and use the fact that for the most part of the segment $[0,2\pi]$ the function $\cos x(t)$ oscillates very fast, since $dx/dt$ is large if only $\abs{\cos t}$ is not very small. It will be shown that this implies that the integral in \eqref{eq:integral} is quite small,
hence for all solutions of \eqref{eq:Joseph} the difference $x(2\pi)-x(0)=\widetilde P_{a,b,\mu}(x(0))-x(0)$ is close to $2\pi a/\mu$.
But if the circle map is uniformly $2\pi\eps$-close to the rigid rotation by the angle $2\pi\alpha$, then its rotation number is
$\eps$-close to $\alpha$. Therefore, inside the $k$-th Arnold tongue $a/\mu$ should be close to $k$, hence $a$ is close to $k\mu$.

For the second theorem, we expand the integral in \eqref{eq:integral} using the formula \eqref{eq:integral} itself:
\begin{equation}\label{eq:dbl-integral}
x(2\pi)-x(0)=\frac{2\pi a}{\mu}+\frac1\mu\int_0^{2\pi}\cos\left(\frac{a\tau+b\sin\tau+\int_0^\tau\cos x(s)\,ds}\mu+x(0)\right)d\tau.
\end{equation}
On the boundary of the Arnold tongue, where either $a=a_{0,k}(b)$ or $a=a_{\pi,k}(b)$, the left-hand side equals $2\pi k$ if either $x(0)=0$ or $x(0)=\pi$.
We will show that the inner integral is small and its influence on the value of the outer integral is also small so it can be dropped.
Then, we replace $a\tau$ with $k\mu\tau$ inside the outer integral (since $a-k\mu$ is small due to Theorem~\ref{thm:asymp-const}).
This yields a change of the outer integral in \eqref{eq:dbl-integral} by the amount of the next order of magnitude. Therefore
\begin{equation*}
2\pi k\approx \frac{2\pi a}{\mu}+\frac1\mu\int_0^{2\pi}\cos\left(k\tau+\frac{b}{\mu}\sin\tau+x(0)\right)d\tau.
\end{equation*}
The integral on the right-hand side can be expressed in terms of $J_k(z)$ by \eqref{eq:bessel-def} and we thus obtain
\begin{equation*}
k\approx \frac{a}{\mu}\pm\frac1\mu J_k(-b/\mu),
\end{equation*}
where the sign is ``$+$'' if $x(0)=0$ and ``$-$'' if $x(0)=\pi$.

\medskip

The remaining part of this paper is organized as follows. In the next section we obtain several estimates for the integral
$\int_0^\tau\cos x(s)\,ds$ and related values. In Section~\ref{sec:proofs} we deduce Theorems~\ref{thm:asymp-const} and~\ref{thm:asymp-bessel}
from these estimates. Finally, in Section~\ref{sec:general} we discuss partial generalizations of these results to the equations of type~\eqref{eq:general}.

\section{Estimations of the integrals}

In what follows in the next section we will need estimates for the integral expressions contained both in \eqref{eq:integral} and \eqref{eq:bessel-def}. Fortunately, these estimates can be obtained simultaneously.
Indeed, consider an equation
\begin{equation}\label{eq:Joseph_bessel_comb}
\frac{dx}{dt}=\frac{\gamma\cos x +a + b \cos t}{\mu}.
\end{equation}
If $\gamma=1$, we obtain the standard Josephson equation \eqref{eq:Joseph}, while if $\gamma=0$ we obtain integrable differential
equation with solutions
\begin{equation*}
x(t)=x(0)+\frac{at+b\sin t}{\mu}.
\end{equation*}
Therefore, if $\hat x(t)$ is the solution of an equation corresponding to $\gamma=0$ with an initial condition $\hat x(0)=0$ , then $\int_0^{2\pi} \cos \hat x(\tau)\,d\tau$ coincides with the integral in~\eqref{eq:bessel-def} for $k=\frac{a}{\mu}$ and $z=-\frac{b}{\mu}$.

Below we always assume that
\begin{equation*}
\abs{\gamma}\le 1.
\end{equation*}

The main instrument in our proof is the following Lemma \ref{lem:avg-diff}. Informally speaking, it states that if $x(t)$ is moving
with almost constant speed, then the time average of a bounded function $\psi$ and its space average along the same arc of a trajectory are close to each other.

\begin{lemma}\label{lem:avg-diff}Suppose that $\dot x(t)$ is of the constant sign for $t\in[t_0,t_1]$.
Denote
\begin{equation*}
\abs{\dot x}_{\min}=\min_{t\in[t_0,t_1]}\abs{\dot x(t)},\quad \abs{\dot x}_{\max}=\max_{t\in[t_0,t_1]}\abs{\dot x(t)},\quad
\osc_{[t_0,t_1]}(\dot x)=\abs{\dot x}_{\max}-\abs{\dot x}_{\min}.
\end{equation*}
Then for any bounded integrable function $\psi$ on a circle we have
\begin{equation}\label{eq:avg-diff}
\left|\frac1{t_1-t_0}\int_{t_0}^{t_1}\psi(x(t))\,dt-\frac1{x_1-x_0}\int_{x_0}^{x_1}\psi(x)\,dx\right|\le
\frac{\osc_{[t_0,t_1]}(\dot x)}{\abs{\dot x}_{\min}}\cdot\norm{\psi}_{C^0},
\end{equation}
where
$x_0=x(t_0)$, $x_1=x(t_1)$, $\norm{\psi}_{C^0}=\sup_{x\in\mathbb R/2\pi\mathbb Z}\abs{\psi(x)}$.
\end{lemma}

\begin{proof}Indeed,
\begin{multline*}
\left|\frac1{t_1-t_0}\int_{t_0}^{t_1}\psi(x(t))\,dt-\frac1{x_1-x_0}\int_{x_0}^{x_1}\psi(x)\,dx\right|={}\\
{}=\left|\frac1{x_1-x_0}\int_{x_0}^{x_1}\left[\frac{x_1-x_0}{t_1-t_0}\cdot\frac{dt}{dx}-1\right]\psi(x)\,dx\right|.
\end{multline*}
It remains to show that the absolute value of the expression in square brackets is not more than $\osc_{[t_0,t_1]}(\dot x)/\abs{\dot x}_{\min}$.
Suppose that $\dot x(t)$ is positive on $[t_0,t_1]$.
Then
$(x_1-x_0)/(t_1-t_0)$ and $dx/dt$ belong to $[\abs{\dot x}_{\min},\abs{\dot x}_{\max}]$, hence
\begin{equation*}
\frac{\abs{\dot x}_{\min}}{\abs{\dot x}_{\max}}-1\le \frac{x_1-x_0}{t_1-t_0}\cdot\frac{dt}{dx}-1\le \frac{\abs{\dot x}_{\max}}{\abs{\dot x}_{\min}}-1,
\end{equation*}
and finally, we obtain that
\begin{equation*}
\left|\frac{x_1-x_0}{t_1-t_0}\cdot\frac{dt}{dx}-1\right|\le \frac{\osc_{[t_0,t_1]}(\dot x)}{\abs{\dot x}_{\min}},
\end{equation*}
so inequality \eqref{eq:avg-diff} is proven. The case of negative $\dot x(t)$ is treated similarly.
\end{proof}

Consider a solution $x(t)$ of equation \eqref{eq:Joseph_bessel_comb} on some interval $[0,t^*]$.
Take all points $0=t_0<t_1<\dots<t_k\le t^*$ such that $x(t_k)\equiv x(0)\pmod{2\pi}$ and
split the interval $[0,t^*]$ by these points into subsegments $I_i=[t_{i-1},t_i]$, $i=1,\dots, k$, and $I^*=[t_k,t^*]$.

As it was said before, the subintervals with ``small'' and ``not so small'' values
of $\abs{\dot x}$ are treated differently. Consider a set
\begin{equation*}
M_\delta=\{\tau\in[0,t]:\abs{\cos \tau}<\delta\},
\end{equation*}
where $\delta$ will be chosen later.

However, from now on we assume that
\begin{subequations}
\label{eq:delta}
\begin{gather}
A:=\abs{a}+1\le \frac{b\delta}{C_a},\label{subeq:delta:A}\\
\delta\ge C_b\sqrt{\frac{\mu}{b}},\label{subeq:delta:b-mu}\\
\delta\le 1\label{subeq:delta:1}
\end{gather}
\end{subequations}
where positive constants $C_a$ and $C_b$ are sufficiently large.

The subsegments $I_i$ and $I^*$ thus fall into the following categories:\\
\emph{type $1$ segments}: the subsegments that are fully covered by $M_\delta$;\\
\emph{type $2$ segments}: the subsegments $I_i$ that are partially covered by $M_\delta$ and $I^*$ in the case if it is not fully covered by $M_\delta$;\\
\emph{type $3$ segments}: the subsegments $I_i$ that are not intersecting with $M_\delta$.\\
Note that there is no more than five segments of type 2 since any such segment is either $I^*$,
or contains one of the four points $\tau$ with $\abs{\cos\tau}=\delta$ in its interior.
Let us also denote by $\calI_1$, $\calI_2$, and $\calI_3$ the union of all the segments of corresponding type.

\medskip

We start with an estimate for the length of the segment of type 2 or 3.

\begin{remark}In the exposition below we use the notation $u(s)=O(v(s))$
in the following precise sense: there exist a constant $C$ such that $\abs{u(s)}\le C v(s)$ (here $v(s)$ is always positive), and this constant
is assumed to be independent from parameters $a$, $b$, $\mu$, from the values of $\delta$, $C_a$, $C_b$ (but we still suppose that \eqref{eq:delta} holds), and from any other variables. Informally speaking, one can fix some large explicit values for $C_a$ and $C_b$ (say, one million) and then replace all $O(\,\cdot\,)$'s in the text below with some explicit estimates. We prefer not to make these hindsight substitutions
in order not to hide dependencies between constants in different estimates.
\end{remark}

\begin{proposition}\label{prop:length-bound}If $C_a$ and $C_b$ in \eqref{eq:delta} are sufficiently large,
then the following holds.

Let $I$ be any segment of type $2$ or $3$. Let $\hat t$ be any point in $I\setminus M_\delta$.
Then the length $\abs{I}$ of this segment satisfies the following estimate:
\begin{equation*}
\abs{I}= O\biggl(\frac{\mu}{b\cos\hat t}\biggr).
\end{equation*}
\end{proposition}

\begin{proof}The proof for type 3 segments is trivial: $x(t)$ travels distance not more than $2\pi$ with its speed bounded from below, thus
the time of the travel is bounded from above. However, for the type 2 segments we need a sort of bootstrapping argument: the lower bound for the speed holds only for initial moment $\hat t$ and it worsens as time goes; nevertheless, it worsens so slowly that we cannot travel distance of $2\pi$ for such a long time that the speed estimate is totally ruined.

Let us pass to the formal proof.
Denote $I=[t_-,t_+]$, $L=\abs{I}=t_+-t_-$.
Inequality  $\abs{x(t_+)-x(t_-)}\le 2\pi$ and the mean value theorem yields that
\begin{equation}\label{eq:len-bound}
\abs{I}\cdot \min_{I} \abs{\dot x}\le \abs{x(t_i)-x(t_{i-1})}\le 2\pi.
\end{equation}
For any $\tau\in I$ we have $\tau=\hat t+s$ for some $s$ with $\abs{s}\le L$. Therefore,
\begin{multline*}
\abs{\dot x(\hat t+s)}\ge \biggl|\frac{b\cos(\hat t+s)}{\mu}\biggr|-\biggl|\frac{\gamma\cos x+a}{\mu}\biggr|\\
{}\ge\frac{b\bigl(\abs{\cos\hat t}-\abs{\cos(\hat t+s)-\cos\hat t}\bigr)}{\mu}-\frac{A}{\mu}\ge\frac{b\abs{\cos\hat t}-A-b\cdot L}\mu
\end{multline*}
since the cosine is a Lipschitz function with constant equal to one.
Now \eqref{eq:len-bound} yields
\begin{equation*}
L\cdot\frac{b\abs{\cos\hat t}-A-b\cdot L}\mu\le 2\pi.
\end{equation*}
The same argument works for any subsegment $\tilde I\subset I$ such that $\hat t\in \tilde I$.
We can choose such $\tilde I$ to be of any length between zero and $L$, so
\begin{equation*}
by^2-(b\abs{\cos\hat t}-A)y+2\pi\mu\ge 0\quad\text{for any $y\in [0,L]$}.
\end{equation*}
Since $\abs{\cos\hat t}>\delta$, one can see that if $C_a\ge 2$ and $C_b\ge 32\pi$ then it follows
from \eqref{subeq:delta:A} and \eqref{subeq:delta:b-mu} that this quadratic polynomial has two positive real roots.
Therefore, $L$ does not exceed its smaller root:
\begin{multline*}
L\le \frac{b\abs{\cos\hat t}-A-\sqrt{(b\abs{\cos\hat t}-A)^2-8\pi b\mu}}{2b}\\
{}=\frac{4\pi \mu}{b\abs{\cos\hat t}-A+\sqrt{(b\abs{\cos\hat t}-A)^2-8\pi b\mu}}\le
\frac{4\pi\mu}{b\abs{\cos\hat t}-A}\le \frac{8\pi\mu}{b\abs{\cos\hat t}}.
\end{multline*}
The last inequality here uses \eqref{subeq:delta:A} with $C_a\ge 2$. The proposition is proven.
\end{proof}

This yields the estimate of a Lebesgue measure of $\calI_1\cup\calI_2$, which we denote by the symbol $\mes(\,\cdot\,)$.
\begin{proposition}\label{prop:int-type12}
If $C_a$ and $C_b$ in \eqref{eq:delta} are sufficiently large,
then
\begin{equation*}
\mes(\calI_1\cup\calI_2)= O\biggl(\dfrac{\mu}{b\delta}+\delta\biggr).
\end{equation*}
\end{proposition}

\begin{proof}The set $\calI_2$ consists of not more than five segments, and the length of each of them is bounded
by Proposition~\ref{prop:length-bound} (we choose $\hat t$ with $\abs{\cos\hat t}=\delta$):
\begin{equation*}
\mes\calI_2= 5\cdot O(\mu / b\delta).
\end{equation*}
The set $\calI_1$ is a subset of $M_\delta$, hence
\begin{equation*}
\mes\calI_1\le \mes M_{\delta}\le 4\arcsin\delta\le 4\cdot\frac{\pi}{2}\delta=O(\delta).\qedhere
\end{equation*}
\end{proof}

Now, let us estimate the integral over any subsegment $I_k$ of type 3.

\begin{proposition}\label{prop:int-type3int}If $C_a$ and $C_b$ in \eqref{eq:delta} are sufficiently large
then for any bounded function $h\colon \mathbb R/2\pi\mathbb Z\to\mathbb R$ with zero average:
\begin{equation*}
\int_0^{2\pi}h(\xi)\,d\xi=0,
\end{equation*}
and for any segment $I_j$ of the type~$3$ we have
\begin{equation*}
\left|\int_{I_j}h(x(\tau))\,d\tau\right|\le
\norm{h}_{C^0}\int_{I_j}\biggl[O\biggl(\frac1{b\abs{\cos\hat t}}\biggr)+
O\biggl(\frac{\mu}{b\cos^2\hat t}\biggr)\biggr]\,d\hat t.
\end{equation*}
\end{proposition}

\begin{proof}It follows from Lemma~\ref{lem:avg-diff} that
\begin{equation}\label{eq:int-type3-est}
\frac1{\abs{I_j}}\left|\int_{I_j}h(x(\tau))\,d\tau\right|\le \norm{h}_{C^0}\cdot \frac{\osc_{I_j}(\dot x)}{\min_{I_j}\abs{\dot x}}.
\end{equation}
Here we use that $x(t_j)-x(t_{j-1})=\pm2\pi$, hence $\int_{x(t_{j-1})}^{x(t_j)}h(x)\,dx=0$.

In order to estimate expressions in the right-hand side of \eqref{eq:int-type3-est}, we take any $\hat t\in I_j$;
Proposition~\ref{prop:length-bound} and Lipschitz property of cosine then give us that
\begin{equation*}
\osc\nolimits_{I_j}(\cos t)\le \abs{I_j}\le O\biggl(\frac{\mu}{b\abs{\cos\hat t}}\biggr).
\end{equation*}
Further,
\begin{align}
\notag\osc\nolimits_{I_j}(\dot x(t))&{}\le \frac{1}{\mu}\bigl(\osc\nolimits_{I_j}\cos x(t)+b\osc\nolimits_{I_j}(\cos t)\bigr)\le
\frac{2}{\mu}+ O\biggl(\frac{1}{\abs{\cos\hat t}}\biggr),\\
\notag\min\nolimits_{I_j}\abs{\dot x(t)}&{}\ge
\frac{1}{\mu}\bigl(\min\nolimits_{I_j}\abs{b\cos t}-A\bigr)\ge
\frac{1}{\mu}\bigl(b\abs{\cos\hat t}-\osc\nolimits_{I_j}(b\cos t)-A\bigr)\\
\label{eq:mindotx}
&=\frac{1}{\mu}\biggl(b\abs{\cos\hat t}-O\biggl(\frac{\mu}{\abs{\cos\hat t}}\biggr)-A\biggr).
\end{align}
For sufficiently large $C_a$ and $C_b$, \eqref{subeq:delta:A} and \eqref{subeq:delta:b-mu}
make the second and the third terms in the right-hand side of \eqref{eq:mindotx} to be smaller than $b\abs{\cos\hat t}/3$,
hence
\begin{equation*}
\min_{I_j}\abs{\dot x(t)}\ge\frac{b\abs{\cos\hat t}}{3\mu}.
\end{equation*}
Therefore,
\begin{equation*}
\frac1{\abs{I_j}}\left|\int_{I_j}h(x(\tau))\,d\tau\right|\le
\norm{h}_{C^0}\cdot O\biggl(\frac{1}{b\abs{\cos\hat t}}+\frac{\mu}{b\cos^2\hat t}\biggr),
\end{equation*}
and it remains to integrate the last inequality over $\hat t\in I_j$.
\end{proof}

\section{Proofs of theorems}
\label{sec:proofs}

\begin{proof}[\proofname\ of Theorem~\textup{\ref{thm:asymp-const}}]
Note that if a circle map is uniformly $2\pi\eps$-close to the rigid rotation by the angle $2\pi a/\mu$, then its rotation number is
$\eps$-close to $a/\mu$. For any solution $x(t)$ of \eqref{eq:Joseph} we have
\begin{equation*}
\left|\frac{x(2\pi)-x(0)}{2\pi}-\frac{a}{\mu}\right|=\left|\frac{1}{2\pi\mu}\int_0^{2\pi}\cos x(t)\,dt\right|,
\end{equation*}
hence the first inequality in \eqref{eq:asymp-const} follows from the next proposition. The second inequality in \eqref{eq:asymp-const} uses
simple estimate $\ln z<2\sqrt{z}$.\end{proof}

\begin{proposition}
\label{prop:asymp-int-const}There exist positive constants $C_1,C_2,K_1,K_2$ such that the following holds.

If parameters $a,b,\mu$ satisfy \eqref{eq:thm-const-cond} then for any $t^*\in [0,2\pi]$ and any solution $x(t)$ of \eqref{eq:Joseph} we have
\begin{equation}\label{eq:asymp-int-const}
\left|\int_0^{t^*}\cos x(t)\,dt\right|\le K_1\sqrt{\frac{\mu}{b}}+\frac{K_2}{b}\ln\biggl(\frac{b}{\mu}\biggr).
\end{equation}
\end{proposition}

\begin{proof}1. Fix values of $C_a$ and $C_b$ such that Propositions~\ref{prop:length-bound}, \ref{prop:int-type12}, and \ref{prop:int-type3int} hold for them. Let us also assume that $C_b\ge 2$.
Set
\begin{equation}\label{eq:C12-from-Cab}
\delta=C_b\sqrt{\frac{\mu}{b}},\quad C_1=\frac{C_b}{C_a},\quad C_2=C_b^2.
\end{equation}
One can see that if $a$, $b$, and $\mu$ satisfy \eqref{eq:thm-const-cond} with these values of $C_1$ and $C_2$ then
all inequalities in \eqref{eq:delta} hold.

2. Split the integral in \eqref{eq:asymp-int-const} into the integrals over subintervals $I_i$ and $I^*$. For the subintervals of types~1 or~2 we use Proposition~\ref{prop:int-type12} and bound the integrand by~$1$. For the subintervals of type~3 we use Proposition~\ref{prop:int-type3int}. Hence
\begin{equation*}
\left|\int_0^{2\pi}\cos x(t)\,dt\right|\le
O\biggl(\dfrac{\mu}{b\delta}+\delta\biggr)+
\int_{\calI_3}
O\biggl(\frac{1}{b\abs{\cos\hat t}}+\frac{\mu}{b\cos^2\hat t}\biggr)\,dt.
\end{equation*}
Since $\calI_3\subset [0,2\pi]\setminus M_\delta$, the last integral is not more than the corresponding integral over $[0,2\pi]\setminus M_\delta$, which equals
\begin{multline}\label{eq:int-calcul}
4\int_0^{\arccos\delta}O\biggl(\frac{1}{b\abs{\cos\hat t}}+\frac{\mu}{b\cos^2\hat t}\biggr)\,d\hat t\\
{}=O\biggl(\frac{1}{b}\biggr)\cdot\ln\frac{1+\sqrt{1-\delta^2}}{\delta}+O\biggl(\frac{\mu}{b}\biggr)\cdot\frac{\sqrt{1-\delta^2}}{\delta}\\
{}\le O\biggl(\frac{1}{b}\ln\frac{2}{\delta}\biggr)+O\biggl(\frac{\mu}{b\delta}\biggr)= O\biggl(\frac{1}{b}\biggl[\ln\sqrt{\frac{b}{\mu}}+\ln\frac{2}{C_b}\biggr]\biggr)+O\biggl(\sqrt{\frac{\mu}{b}}\biggr).
\end{multline}
As $C_b\ge 2$, the second term in square brackets is negative and can be discarded. This yields \eqref{eq:asymp-int-const}.
\end{proof}

\begin{proof}[\proofname\ of Theorem~\textup{\ref{thm:asymp-bessel}}]
The proof contains two parts. The core part (see items 2--6 below) shows that if some conditions similar to those of Theorem~\ref{thm:asymp-const} hold for $a=a_{0,k}(b,\mu)$ (or $a=a_{\pi,k}(b,\mu)$), $b$, and $\mu$, then $a$ is close to the Bessel function as stated in \eqref{eq:asymp-bessel}.
However, \emph{a priori} we do not know that for a given $b$ and $\mu$ the boundaries $a_{\dots,k}(b,\mu)$ of $k$-th Arnold tongue satisfy these estimates. Thus we start with preliminary part (item 1 below) showing that under some conditions on $k$, $b$, and $\mu$ the triples $(a_{0,k}(b,\mu),b,\mu)$ and $(a_{\pi,k}(b,\mu),b,\mu)$ satisfy conditions needed for the core part of the proof.

1. First of all, fix $C_a$ and $C_b$ such that Propositions~\ref{prop:length-bound}, \ref{prop:int-type12}, and \ref{prop:int-type3int} hold for them. Now fix values of $C_1$, $C_2$ and $\delta=C_b\sqrt{b/\mu}$ defined by \eqref{eq:C12-from-Cab}.

Let us show that if $C_1'$ and $C_2'$ are appropriately chosen, then for any $b$, $\mu$, and $k$ that satisfy \eqref{eq:thm-bessel-cond}, each one of the triples
\begin{equation}\label{eq:triples}
(k\mu,b,\mu),\quad (a_{0,k}(b,\mu),b,\mu),\quad (a_{\pi,k}(b,\mu),b,\mu)
\end{equation}
satisfies \eqref{eq:thm-const-cond}. For the first triple this obviously holds for any $C_1'\le C_1$, $C_2'\ge C_2$.
Consider the second triple (the argument for the third one is exactly the same).
If $C_1'$ is sufficiently small and $C_2'$  is sufficiently large, then the following inequalities hold:
\begin{equation}\label{eq:C12prime}
\frac{K_1}{\sqrt{C_2'}}+K_2C_1'< 1,\quad C_1'\le C_1/2,\quad C_2'\ge C_2,
\end{equation}
Take any constants $C_{1}'$ and $C_{2}'$ that satisfy \eqref{eq:C12prime}.
We now show that for any $b,\mu,k$ satisfying \eqref{eq:thm-bessel-cond} we have
\begin{equation}\label{eq:diff-a-kmu}
\abs{a_{0,k}(b,\mu)-k\mu}< 1.
\end{equation}
Indeed, the inequality \eqref{eq:diff-a-kmu} holds for all sufficiently large $b$ due to Theorem~\ref{thm:asymp-const}.
Therefore, if it fails for some $b'$, $\mu'$, and $k'$ satisfying \eqref{eq:thm-bessel-cond} then by continuity there exist $b''\ge b'$ such that
\eqref{eq:diff-a-kmu} ``almost holds'': $|a_{0,k'}(b'',\mu')-k'\mu'|=1$. Clearly, the triple $(b'',\mu',k')$ also satisfies \eqref{eq:thm-bessel-cond}, and the triple $(a_{0,k'}(b'',\mu'),b'',\mu')$ satisfies conditions \eqref{eq:thm-const-cond} of Theorem~\ref{thm:asymp-const} because
\begin{equation*}
\abs{a_{0,k'}(b'',\mu')}+1\le \abs{k'\mu'}+2\le 2C_1'\sqrt{b''\mu'}\le C_1\sqrt{b''\mu'}.
\end{equation*}
Therefore, Theorem~\ref{thm:asymp-const} yields
\begin{equation*}
\abs{a_{0,k'}(b'',\mu')-k'\mu'}\le K_1\sqrt{\frac{\mu'}{b'}}+\frac{K_2}{\sqrt{b'\mu'}}\le \frac{K_1}{\sqrt{C_2'}}+K_2C_1'< 1,
\end{equation*}
this contradicts our assumption $|a_{0,k'}(b'',\mu')-k'\mu'|=1$.

2. From now on we fix $C_{1,2}'$ that satisfy \eqref{eq:C12prime}. In particular this means
that Propositions~\ref{prop:length-bound}, \ref{prop:int-type12}, and \ref{prop:int-type3int} hold for all triples in \eqref{eq:triples},
where $b$, $\mu$, and $k$ satisfy~\eqref{eq:thm-bessel-cond}.

Consider a point $a_{0,k}(b,\mu)$ with such values of $b$, $\mu$, and $k$.
Let $x_0(t)$ be the solution of \eqref{eq:Joseph} with $a=a_{0,k}(b,\mu)$ such that $x_0(0)=0$.
As it was said before, then $x_0(2\pi)-x_0(0)=2\pi k$, and \eqref{eq:integral} yields
\begin{equation*}
k=\frac{x_0(2\pi)-x_0(0)}{2\pi}=\frac{a_{0,k}(b,\mu)}{\mu}+\frac1\mu\int_0^{2\pi}\cos x_0(\tau)\,d\tau.
\end{equation*}
Therefore,
\begin{equation}\label{eq:bessel-differ}
a_{0,k}(b,\mu)-k\mu+J_k\Bigl(-\frac b\mu\Bigr)=
-\frac1{2\pi}\int_0^{2\pi}\cos\Bigl(kt+\frac{b}{\mu}\sin t+\psi(t)\Bigr)-\cos\Bigl(kt+\frac{b}{\mu}\sin t\Bigr)\,dt,
\end{equation}
where
\begin{equation}\label{eq:psi}
\psi(t)=\biggl(\frac{a_{0,k}(b,\mu)}\mu -k\biggr)t+\frac1\mu\int_0^t\cos x_0(\tau)\,d\tau.
\end{equation}
Denote also $\hat x(t)=kt+(b/\mu)\sin t$. Then the right-hand side in \eqref{eq:bessel-differ} equals
\begin{equation*}
-\frac1{2\pi}\int_0^{2\pi}\cos(\hat x(t))\cdot(\cos\psi(t)-1)\,dt
+\frac1{2\pi}\int_0^{2\pi}\sin(\hat x(t))\cdot\sin\psi(t)\,dt.
\end{equation*}
Denote the summands here as $S_1$ and $S_2$ respectively.

3. Let us start by estimating the norm of $\psi$. The triple $(a_{0,k}(b,\mu),b,\mu)$ satisfies conditions
\eqref{eq:thm-const-cond}, hence we may apply Theorem~\ref{thm:asymp-const} for the first summand in \eqref{eq:psi} and Proposition~\ref{prop:asymp-int-const} for the second one.
Then we obtain
\begin{equation}\label{eq:norm-psi}
\norm{\psi}_{C^0}=O\biggl(\frac{1}{\sqrt{b}}\biggl(\frac{1}{\mu^{1/2}}+\frac{1}{\mu^{3/2}}\biggr)\biggr),
\end{equation}
In order to estimate $S_1$, we bound the first cosine by $1$ and the second multiplier by $\norm{\psi}_{C^0}^2/2$. This yields
\begin{equation*}
\abs{S_1}=O\biggl(\frac{1}{b}\biggl(\frac{1}{\mu}+\frac{1}{\mu^3}\biggr)\biggr).
\end{equation*}

4. The estimation of $S_2$ goes along the lines of proof of Proposition~\ref{prop:asymp-int-const}.
We split $[0,2\pi]$ into subsegments $J_j$ and $J^*$ by the points where $\hat x(t)\equiv0\pmod{2\pi}$,
consider the set $M_\delta$, and classify these subsegments into types 1, 2, or 3 as above.

Recall that $\hat x(t)$ is a solution of the equation \eqref{eq:Joseph_bessel_comb} with $\gamma=0$, and the parameters equal to $\hat a=k\mu$, $\hat b=b$, $\hat\mu=\mu$. As it was said before, we can apply Propositions~\ref{prop:length-bound}, \ref{prop:int-type12}, and \ref{prop:int-type3int} to it.

The integral in $S_2$ splits into the sum of integrals over subintervals $J_j$ and $J^*$. We denote the part of this sum
corresponding to the segments of types 1 and 2 by $S_2^{(1,2)}$ and the part corresponding to the segments of type 3 by $S_2^{(3)}$.
Proposition~\ref{prop:int-type12} applies to $S_2^{(1,2)}$:
\begin{multline*}
\abs{S_2^{(1,2)}}\le \sum_{\substack{J=J_j, J^*\\\text{ of types 1 or 2}}}\left|\frac1{2\pi}\int_J\sin\hat x(t)\cdot\sin\psi(t)\,dt\right|\le
\frac{\norm{\sin\psi}_{C^0}}{2\pi}\sum_{\substack{J=J_j, J^*\\\text{ of types 1 or 2}}}\abs{J}\\
{}\le \biggl(\frac{1}{\sqrt{b}}\biggl(\frac{1}{\mu^{1/2}}+\frac{1}{\mu^{3/2}}\biggr)\biggr)\cdot O\biggl(\frac{\mu}{b\delta}+\delta\biggr)=
O\biggl(\frac{1}{b}\biggl(1+\frac{1}{\mu}\biggl)\biggr).
\end{multline*}

5. The part $S_2^{(3)}$ is estimated as follows. Fix any point $t_j$ in each $I_j$. Then
\begin{align*}
\abs{S_2^{(3)}}\le {}&\sum_{J_j\text{ of type 3}}\left|\frac1{2\pi}\int_{J_j}\sin \hat x(t)\cdot\sin\psi(t_j)\,dt\right|\\
{}+{}&\sum_{J_j\text{ of type 3}}\left|\frac1{2\pi}\int_{J_j}\sin \hat x(t) \cdot\bigl[\sin\psi(t)-\sin\psi(t_j)\bigr]\,dt\right|.
\end{align*}
Denote the two sums on the right-hand side by $S_2^{(3)\text{\raisebox{0.3ex}{$\star$}}}$ and $S_2^{(3)\text{\raisebox{0.3ex}{$\star\star$}}}$, respectively.
The first sum, $S_2^{(3)\text{\raisebox{0.3ex}{$\star$}}}$ is estimated by Proposition~\ref{prop:int-type3int}:
\begin{equation*}
S_2^{(3)\text{\raisebox{0.3ex}{$\star$}}}
\le \norm{\psi}_{C^0}\int_{\mathcal J_3} \biggl[O\biggl(\frac1{b\abs{\cos \hat t}}\biggr)+O\biggl(\frac{\mu}{b\cos^2\hat t}\biggr)\biggr]\,d\hat t.
\end{equation*}
The integral is managed exactly in the same way as the integral over $\mathcal I_3$ in the proof of Proposition~\ref{prop:asymp-int-const};
together with inequality $\ln z\le 2\sqrt{z}$ and \eqref{eq:norm-psi} this yields
\begin{equation*}
S_2^{(3)\text{\raisebox{0.3ex}{$\star$}}}=O\biggl(\frac{1}{b}\biggl(1+\frac{1}{\mu^2}\biggr)\biggr).
\end{equation*}

6. In the sum $S_2^{(3)\text{\raisebox{0.3ex}{$\star\star$}}}$ we bound $\sin\hat x(t)$ by~$1$ and the difference in square brackets by $\osc_{J_j}\psi\le \abs{J_j}\cdot\max_{J_j}\abs{\psi'}\le \abs{J_j}\cdot(\abs{a-k\mu}+1)/\mu$:
\begin{equation*}
S_2^{(3)\text{\raisebox{0.3ex}{$\star\star$}}}\le \sum_{J_j\text{ of type 3}}\abs{J_j}\osc\nolimits_{J_j}\psi\le
\sum_{J_j\text{ of type 3}}\abs{J_j}^2\cdot\biggl(\biggl|\frac{a}{\mu}-k\biggr|+\frac{1}{\mu}\biggr).
\end{equation*}
We have already seen in \eqref{eq:diff-a-kmu} that $\abs{a-k\mu}=O(1)$ hence the last bracket is $O(1/\mu)$.
Proposition~\ref{prop:length-bound} yields
\begin{equation*}
\abs{J_j}^2\le \int_{J_j}O\biggl(\frac{\mu}{b\abs{\cos\hat t}}\biggr)\,d\hat t,
\end{equation*}
therefore by \eqref{eq:int-calcul} we obtain
\begin{equation*}
S_2^{(3)\text{\raisebox{0.3ex}{$\star\star$}}}\le \int_{[0,2\pi]\setminus M_\delta}O\biggl(\frac{d\hat t}{b\abs{\cos\hat t}}\biggr)=
O\biggl(\frac{\ln(b/\mu)}{b}\biggr).
\end{equation*}

Joining together the estimates for $S_1$, $S_2^{(1,2)}$, $S_2^{(3)\text{\raisebox{0.3ex}{$\star$}}}$, and $S_2^{(3)\text{\raisebox{0.3ex}{$\star\star$}}}$, we complete the proof.
\end{proof}

\section{Generalizations}
\label{sec:general}

Let us now discuss some possible generalizations of Theorems \ref{thm:asymp-const} and \ref{thm:asymp-bessel}.
Theorem~\ref{thm:asymp-const} can be straightforwardly generalized to any equation of the form \eqref{eq:general} such that
the graph of the function $g$ transversely crosses the line $\{t=0\}$.
More precisely, the proof given above uses only the following properties of the functions $f$ and $g$:
\begin{enumerate}
\item functions $f$ and $g$ are bounded by $1$;
\item $g$ is Lipschitz with constant $1$;
\item the graph $y=g(t)$ transversely intersects the line $y=0$.
\end{enumerate}
(Recall that also $\int_{0}^{2\pi}f(x)\,dx=0$, $\int_{0}^{2\pi}g(t)\,dt=0$.)

Constants equal to one in these properties can be easily replaced by any other constants by the means of the substitutions
\begin{align*}
(f,g,a,b,\mu)&{}\to (f/D, g/D, a/D, b, \mu/D),\\
(f,g,a,b,\mu)&{}\to (f, g/D, a, bD, \mu)
\end{align*}
with some $D>0$. As for the last condition, it is used in two parts of the proof:
(1) estimates of $\mes M_\delta$ and (2) estimates of the integrals
$\int_{[0,2\pi]\setminus M_\delta} d\hat t/\abs{g(\hat t)}$ and
$\int_{[0,2\pi]\setminus M_\delta} d\hat t/g^2(\hat t)$ in \eqref{eq:int-calcul}. Let us express transversality condition in the following
quantitative way: there exists $\varepsilon_0>0$ and $L>0$ such that for any $\varepsilon\le \varepsilon_0$ we have
\begin{equation*}
\mes M_\varepsilon:=\mes\{t:\abs{g(t)}\le\varepsilon\}\le L\varepsilon.
\end{equation*}
Suppose that $\delta\le \varepsilon_0$ (this is a required modification of condition \eqref{subeq:delta:1}), then
$\mes M_\delta$ is estimated exactly in the same way as in the proof, and for integrals we use the following estimate:
\begin{multline*}
\int_{[0,2\pi]\setminus M_\delta}\frac{d\hat t}{g^2(\hat t)}=
\int_{0}^\infty \mes\biggl\{\hat t\in [0,2\pi]\setminus M_\delta: \frac{1}{g^2(\hat t)}\ge y\biggr\}\, dy\\
{}=\int_{0}^\infty \mes\biggl\{\hat t\in [0,2\pi]: \delta\le g(\hat t)\le \frac{1}{\sqrt{y}}\biggr\}\, dy.
\end{multline*}
The set is empty if $y>1/\delta^2$, otherwise we bound its measure by $\mes M_{1/\sqrt{y}}$, which is estimated via transverality condition:
\begin{equation*}
\int_{0}^{1/\delta^2} \mes M_{1/\sqrt{y}}\, dy\le
\int_{0}^{1/\varepsilon_0^2} 2\pi\, dy+\int_{1/\varepsilon_0^2}^{1/\delta^2} \frac{L}{\sqrt{y}} dy\le
O(1)+O\biggl(\frac{1}{\delta}\biggr).
\end{equation*}
Another integral is bounded similarly, and \eqref{eq:int-calcul} preserves its form. Therefore, we obtain the following generalization of Theorem~\ref{thm:asymp-const}.

\begin{theorem}\label{thm:asymp-const-gen}Fix any positive constants $L_0$, $L_1$, $L_2$, $L_3$. Then
there exist positive constants $C_1,C_2,K_1,K_2$ depending on $L_{0,1,2,3}$ such that the following holds.
Consider any functions $f$ and $g$ with zero averages such that
\begin{enumerate}
\item their continuous norms are bounded: $\norm{f}_{C_0}\le L_1$, $\norm{g}_{C_0}\le L_1$,
\item $g$ is Lipschitz with constant $L_2$: $\abs{g(t_1)-g(t_2)}\le L_2\abs{t_1-t_2}$,
\item for any $\delta<1/L_0$ there is a bound $\mes\{\abs{g(t)}<\delta\}\le L_3\delta$.
\end{enumerate}
Then if the parameters $a,b,\mu$ of the equation \eqref{eq:general} are such that
\begin{equation*}
|a|+1\le C_1\sqrt{b\mu},\qquad
b\ge C_2 \mu
\end{equation*}
we have
\begin{equation*}
\left|\frac a\mu-\rho_{a,b,\mu}\right|\le \frac{K_1}{\sqrt{b\mu}}+\frac{K_2}{b\mu}\ln\biggl(\frac{b}{\mu}\biggr)\le
\frac{K_1}{\sqrt{b\mu}}+\frac{2K_2}{\sqrt{b\mu^3}}.
\end{equation*}
\end{theorem}

As for Theorem~\ref{thm:asymp-bessel}, we have seen in Section~\ref{sec:intro} that the reduction to a Riccati equation and
identification of fixed point of $\widetilde P_{a,b,\mu}$ for the Arnold tongue boundaries with $0$ and $\pi$ works only if $f(x)=\cos x$ and
$g(t)$ is even. These conditions cannot be significantly extended (trivial extension is obtained by coordinate change $x'=x+x_0$, $t'=t+t_0$; the conditions take form $f(x')=\cos(x'-x_0)$, $g(t')=g(2t_0-t')$). Under these assumptions and transversality condition discussed above
the following analogue of Theorem~\ref{thm:asymp-bessel} holds. Modifications in its proof are exactly the same as above.

\begin{theorem}\label{thm:asymp-bessel-gen}
Fix any positive constants $L_0$, $L_1$, $L_2$, $L_3$.
Then there exist positive constants $C_1',C_2',K_1',K_2',K_3'$ depending on $L_{0,1,2,3}$ such that the following holds.

Consider any function $g$ with zero average that satisfies conditions 1--3 of Theorem~\ref{thm:asymp-const-gen} and the condition $g(t)=g(-t)$.
Let $a_{0,k}(b,\mu)$ and $a_{\pi,k}(b,\mu)$ be the boundaries of $k$-th Arnold tongue of the equation \eqref{eq:general} with this $g$ and $f(x)=\cos x$. Then if the parameters $b,\mu$ and a number $k \in \mathbb{Z}$ satisfy inequalities
\begin{equation*}
|k\mu|+1\le C_1'\sqrt{b\mu},\qquad
b\ge C_2' \mu
\end{equation*}
the following estimates hold
\begin{equation*}
\begin{aligned}
\left|\frac {a_{0,k}(b)}{\mu}-k+\frac{1}{\mu}\tilde J_k\left(-\frac{b}{\mu}\right)\right|&{}\le \frac{1}{b}\biggl(K_1'+\frac{K_2'}{\mu^3}+K_3'\ln\biggl(\frac{b}{\mu}\biggr)\biggr),\\
\left|\frac {a_{\pi,k}(b)}{\mu}-k-\frac{1}{\mu}\tilde J_k\left(-\frac{b}{\mu}\right)\right|&{}\le \frac{1}{b}\biggl(K_1'+\frac{K_2'}{\mu^3}+K_3'\ln\biggl(\frac{b}{\mu}\biggr)\biggr),
\end{aligned}
\end{equation*}
where
\begin{equation*}
\tilde J_k(-z)=\frac{1}{2\pi}\int_0^{2\pi}\cos(kt+zG(t))\,dt,\qquad G(t)=\int_0^t g(\tau)\,d\tau.
\end{equation*}
\end{theorem}

The function $\tilde J_k$ stems from integral representations \eqref{eq:integral} \eqref{eq:dbl-integral}, which now have the form
\begin{align*}
x(t)-x(0)&{}=\frac{at+bG(t)+\int_0^t\cos x(\tau)\,d\tau}{\mu},\\
x(2\pi)-x(0)&{}=\frac{2\pi a}{\mu}+\frac1\mu\int_0^{2\pi}\cos\left(\frac{a\tau+bG(\tau)+\int_0^\tau\cos x(s)\,ds}\mu+x(0)\right)d\tau
\end{align*}
(note that $G(2\pi)=0$ due to \eqref{eq:fg-averages}). The function $\tilde J_k$ also has asymptotic representation similar to the one for $J_k$:
\begin{multline}\label{eq:asymp-gen-bessel}
\tilde J_k(-z)\sim \sum_j \frac1{\sqrt{2\pi z\abs{g'(t_j)}}}\cos\biggl(zG(t_j)+kt_j+\frac{\pi}{4}\operatorname{sgn}(g'(t_j))\biggr)\\
\text{as $z\to+\infty$},
\end{multline}
where the sum is taken over all the zeros $t_j$ of the function $g$ on a circle.

Recall that these zeroes are simple (and hence the denominators in \eqref{eq:asymp-gen-bessel} are nonzero) due to transversality condition 3 of Theorems \ref{thm:asymp-const-gen} and \ref{thm:asymp-bessel-gen}.

\paragraph{Acknowledgements.} We would like to thank V.~Kleptsyn and I.~Schurov for helpful conversations and corrections. We would like also to thank I.~Schurov for providing us with the results of computer simulations used in Figure \ref{fig:graph}.

\end{document}